\documentclass{amsart}%
\usepackage{amssymb,latexsym,amsmath,amscd,amsthm,amsfonts}
\usepackage{amsmath}
\usepackage{amsfonts}
\usepackage{amssymb}
\usepackage{graphicx}%
\providecommand{\U}[1]{\protect\rule{.1in}{.1in}}
\begin{document}
\title{Spectra for compact quantum group coactions and crossed products}
\author{Raluca Dumitru, Costel Peligrad}
\address{Raluca Dumitru: Department of Mathematics and Statistics, University of North
Florida, 1 UNF Drive, Jacksonville, Florida 32224; Institute of Mathematics of
the Romanian Academy, Bucharest, Romania; E-mail address: raluca.dumitru@unf.edu\\
Costel Peligrad: Department of Mathematical Sciences, University of
Cincinnati, 610A Old Chemistry Building, Cincinnati, OH 45221; E-mail address: costel.peligrad@uc.edu}
\thanks{The first author was supported by a UNF Summer Research Grant}
\maketitle

\begin{abstract}
We present definitions of both Connes spectrum and strong Connes spectrum for
actions of compact quantum groups on C*-algebras and obtain necessary and
sufficient conditions for a crossed product to be a prime or a simple
C*-algebra. Our results extend to the case of compact quantum actions the
results in \cite{glp} which in turn, generalize results by Connes, Olesen and
Pedersen and Kishimoto for abelian group actions. We prove in addition that
the Connes spectra are closed under tensor products. These results are new for
compact nonabelian groups as well.

\end{abstract}

\ 

\ 

\newtheorem{definition}{Definition}[section] \newtheorem{lemma}[definition]{Lemma}

\newtheorem{proposition}[definition]{Proposition}
\newtheorem{theorem}[definition]{Theorem}
\newtheorem{corollary}[definition]{Corollary} \theoremstyle{definition} \newtheorem{remark}[definition]{Remark}

\section{Introduction}

In his fundamental paper \cite{connes}, Connes defines the invariant $\Gamma$
called, in his name, the Connes spectrum, for abelian group actions on von
Neumann algebras. Among other results, he obtained conditions for a crossed
product to be a factor. Subsequently, Olesen and Pedersen \cite{pedersen} have
defined the Connes spectrum for abelian group actions on C*-algebras. They
have extended the results of Connes to the case of crossed products of
C*-algebras by abelian group actions obtaining conditions that such a crossed
product be a prime C*-algebra. However, the similar result for simple crossed
products using the Olesen-Pedersen version of Connes spectrum is false. In
\cite{kishimoto}, Kishimoto has shown that the result is true for simple
crossed products if his "strong Connes spectrum" is used instead of Olesen and
Pedersen' Connes spectrum. Rieffel \cite{rieffel} and Landstad \cite{landstad}
have put the problem of finding a "good" definition of the Connes spectrum for
compact, not necessarily abelian group actions on C*-algebras. In
\cite{landstad}, Landstad remarks that a "good" definition of the Connes
spectrum should lead to a result that generalizes the Olesen-Pedersen
characterization of prime crossed products to the case of nonabelian compact
group actions. Gootman, Lazar, Peligrad \cite{glp} have defined the Connes
spectrum and the strong Connes spectrum for compact, not necessarily abelian
group actions on C*-algebras. In the case of abelian groups,these notions
coincide with the previous ones. Moreover, in \cite{glp}, Gootman, Lazar, and
Peligrad prove the characterizations of the primeness and simplicity of
crossed products using their definitions. In this paper we present definitions
of both the Connes spectrum (Definition \ref{arvconnesspectra}) and the strong
Connes spectrum (Definition \ref{strongspectra}) in the case of compact
quantum groups and prove the corresponding characterizations of primeness and
simplicity of crossed products (Theorems \ref{primeness} and \ref{simple}). In
addition, we prove that the Connes spectra are closed under tensor products
(Propositions \ref{strongconnesclosed} and \ref{connesclosed} ). This result
is new for nonabelian compact groups as well. We use the techniques developed
by Woronowicz \cite{wor1,wor2}, Boca \cite{boca}, and the authors in
\cite{ral, ralpel}. In addition, this paper contains new methods for the study
of the hereditary C*-subalgebras that are invariant under a compact quantum
group coaction (Section 3) and for the proof of the key Lemma
\ref{quantum2.10}.

\section{Preliminaries}

Let $G=(A,\Delta)$ be a compact quantum group (see \cite{wor2}) and let
$(B,G,\delta)$ be a quantum dynamical system, where $B$ is a C*-algebra and
$\delta$ is a coaction of $A$ on $B$ (see \cite{boca} or \cite{podles}).
Denote by $\widehat{G}$ the set of all equivalence classes of irreducible
representations of $G$ (\cite{wor1}, section 4).

For each $\alpha\in\widehat{G}$, $u^{\alpha}=\underset{i,j=1}{\overset
{d_{\alpha}}{\sum}}m_{ij}^{\alpha}\otimes u^{\alpha}_{ij}$ denotes a
representative of each class. Then the linear space generated by $\{u^{\alpha
}_{ij}|\alpha\in\widehat{G},1\leq i,j \leq d_{\alpha}\}$ is a $\ast$-algebra
$\mathcal{A}$, called the Woronowicz-Hopf algebra (\cite{wor2}, Section 5).
For $\alpha\in\widehat{G}$ and $u^{\alpha}\in\alpha$ a unitary representative,
denote $\overline{u^{\alpha}}=\underset{i,j=1}{\overset{d_{\alpha}}{\sum}%
}m_{ij}^{\alpha}\otimes{u^{\alpha}_{ij}}^{\ast}$. Then $\overline{u^{\alpha}}$
is a (not necessarily unitary) representation of A called the adjoint of
$u^{\alpha}$. We will denote by $\overline{\alpha}$ the equivalence class of
$\overline{u^{\alpha}}$.

Set $\chi_{\alpha}=\underset{i=1}{\overset{d_{\alpha}}{\sum}}u^{\alpha}_{ii}$.
We use $F_{\alpha}$ to denote the unique positive, invertible operator in
$B(H_{\alpha})$, that intertwines $u^{\alpha}$ with its double contragradient
representation $(u^{\alpha})^{cc}$ such that $tr(F_{\alpha})=tr(F_{\alpha
}^{-1})$. Set $M_{\alpha}=tr(F_{\alpha})$ (\cite{wor1}, Theorem 5.4).

Since every positive matrix is unitarily equivalent to a diagonal matrix, we
may assume that the matrices $F_{\alpha}$ are diagonal: $F_{\alpha
}=diag\{f_{1}^{\alpha},\ldots,f_{d_{\alpha}}^{\alpha}\}$. The formula
$f_{1}^{\alpha}(u^{\alpha}_{nm})=\delta_{nm}f_{m}$ defines a linear form on
$\mathcal{A}$. The above assumption implies that all $u^{\alpha}_{ij}$ are
mutually orthogonal in $H_{h}$ and therefore%

\begin{equation}
\label{diag}h({u^{\alpha}_{ij}}^{\ast}u^{\alpha}_{mn})=\frac{1}{M_{\alpha}%
}\frac{1}{f_{i}^{\alpha}}\delta_{im}\delta_{jn} \text{\; and}%
\end{equation}

\begin{equation}
\label{diag2}h({u^{\alpha}_{mk}}{u^{\alpha}_{nl}}^{\ast})=\frac{f_{l}^{\alpha
}}{M_{\alpha}}\delta_{mn}\delta_{lk}%
\end{equation}

where $h$ is the Haar state on $G$ and $\delta_{rs}$ are the Kronecker
$\delta$'s (\cite{wor1}, Theorem 5.7).

For each $\alpha\in\widehat{G}$, denote by $B_{\alpha}^{\delta}$ the
associated spectral subspace defined by (see \cite{ralpel} or \cite{boca}):
\[
B_{\alpha}^{\delta}=\{P_{\alpha}(x)|\; x\in B\},
\]
where $P_{\alpha}=(\iota\otimes h_{\alpha})(\delta(x))$ and $h_{\alpha
}=M_{\alpha}h\cdot(\chi_{\alpha}\ast f_{1}^{\alpha})^{\ast}$. Recall that for
all a, b $\in$ $A$, $h\cdot a (b)=h(ba)$ and for all linear functionals $\xi$
on $A$, $a\ast\xi=(\xi\otimes\iota)(\Delta(a))$ (see \cite{wor2} relation
1.14, or \cite{ral}). In particular, for $\alpha=\iota$, $P_{\iota}$ is the
projection of $B_{\ast}$ onto the fixed point algebra $B^{\delta}$.

Let $c_{ij}=M_{\alpha}(u^{\alpha}_{ij}\ast f_{1}^{\alpha})^{\ast}$. Note that,
since $F_{\alpha}$ is a diagonal matrix we obtain $u^{\alpha}_{ij}\ast
f_{1}=f_{i}^{\alpha}u^{\alpha}_{ij}$ and hence $c_{ij}=M_{\alpha}f_{i}%
^{\alpha}{u^{\alpha}_{ij}}^{\ast}$.

Define the mapping $P^{\alpha}_{ij}:B \to B$ by
\[
P^{\alpha}_{ij}(x)=(id\otimes h \cdot c^{\alpha}_{ji})(\delta(x)),
\]
for all $x \in B$. Note that $P^{\alpha}_{ij}P^{\alpha}_{kl}=\delta
_{il}P^{\alpha}_{kj}$.

For each $\alpha\in\widehat{G}$, define%

\[
B_{2}^{\delta}(u^{\alpha} )=\{[P^{\alpha}_{ij}(x)]_{ij} |\; x\in B\}\subseteq
B\otimes\mathcal{M}_{d_{\alpha}},
\]
where $[P^{\alpha}_{ij}(x)]_{ij}=\underset{i,j=1}{\overset{d_{\alpha}}{\sum}}
P^{\alpha}_{ij}(x)\otimes m_{ij}^{\alpha}$ with $\{ m_{ij}^{\alpha}|1\leq
i,j\leq d_{\alpha}\} $ the set of elementary matrices in the algebra
$\mathcal{M}_{d_{\alpha}}$ of scalar matrices of order $d_{\alpha}\times
d_{\alpha}$.

Notice that $B_{2}^{\delta}(u^{\alpha})$ depends on the representative
$u^{\alpha}$, not only on the equivalence class $\alpha\in\widehat{G}$.
However, for two equivalent representations $u_{1}^{\alpha}$ and
$u_{2}^{\alpha}$, the corresponding $B_{2}^{\delta}$ are spatially isomorphic.

The proofs of the following remarks are straightforward from definitions.

\begin{remark}\label{rem2.1}
\begin{itemize}
\item[(1)] If $u^{\alpha}$ is a unitary representation of $G$, then
$$\delta(P_{ij}^{\alpha}(x))=\underset{k=1}{\overset{d_{\alpha}}{\sum}}P_{ik}^{\alpha}(x)\otimes u_{kj}^{\alpha}.$$
\item[(2)] $B_{\alpha}^{\delta}=linspan\{P_{ij}^{\alpha}(x)|\; x\in B,\; i,j=1,2,\ldots , d_{\alpha}\}.$
\item[(3)] For $x\in B$, let $X=[P_{ij}^{\alpha}(x)]_{ij}=\underset{i,j=1}{\overset{d_{\alpha}}\sum}P_{ij}^{\alpha}(x)\otimes m_{ij}^{\alpha}$. Then $X\in B_{2}^{\delta}(u^{\alpha})$ and $\delta_{13}(X)=(X\otimes 1_{A})(1_{B}\otimes u^{\alpha})$, where $1_{A}$ is the unit of $A$ and $1_{B}$ is the unit of the multiplier algebra of $B$. The leg numbering notation used here is the standard one (\cite{baaj} and \cite{wor1}). Also, $B_{2}^{\delta}(u^{\alpha})$ is isomorphic as a Banach space to $B_{\alpha}^{\delta}$ through the mapping $X\to \underset{i=1}{\overset{d_{\alpha}}{\sum}}P_{ii}^{\alpha}(x)$. Therefore:
$$B_{2}^{\delta}(u^{\alpha})=\{X \in B\otimes \mathcal{M}_{d_{\alpha}}| \; \delta_{13}(X)=(X\otimes 1_{A})(1_{B}\otimes u^{\alpha})\}$$
\item[(4)]Let $x\in B$ and fix $x_{i_{0}j_{0}}=P_{i_{0}j_{0}}^{\alpha}(x)$. Then $x_{i_{0}j_{0}}\in B$ and $[P_{ij}^{\alpha}(x_{i_{0}j_{0}})]_{ij}\in B_{2}^{\delta}(u^{\alpha})$ is a matrix whose only non-zero row is the $j_{0}$-row and whose $j_{0}j$-entry is given by $P_{i_{0}j}^{\alpha}(x)$, for each $j=1,2,\ldots , d_{\alpha}$. Furthermore,
$$B_{2}^{\delta}(u^{\alpha})=linspan \{[P_{ij}^{\alpha}(x_{rs})]_{ij}|\; r,s=1,2,...,d_{\alpha}\}.$$
\end{itemize}
\end{remark}

With $\alpha\in\widehat{G}$ and $v$ the right regular representation of $G$,
we will use the following notations (see \cite{wor2} relation 3.2 and
\cite{ral}):
\[
p_{\alpha}=(id\otimes h_{\alpha})(v^{\ast}),
\]
\[
\mathcal{F}_{v}(a)=(id\otimes ha)(v^{\ast}).
\]

Denote by $\widehat{A}$ the norm closure of the set of all operators of the
form $\mathcal{F}_{v}(a)$, where $a\in A$.

Recall that the crossed product $B\times_{\delta}G$ is defined to be the
C*-algebra generated by all elements of the form $(\pi_{u}\times\pi
_{h})(\delta(b))(1\otimes\mathcal{F}_{v}(a))$, where $a\in A$, $b\in B$,
$\pi_{u}:B\to B(H_{u})$ is the universal representation of the C*-algebra $B$
and $\pi_{h}:A\to B(H_{h})$ is the GNS representation of $A$ associated to the
Haar state $h$.

Furthermore, if $\alpha_{1}, \alpha_{2}\in\widehat{G}$, define
\[
S_{\alpha_{1},\alpha_{2}}=(1\otimes p_{\alpha_{1}})(B\times_{\delta
}G)(1\otimes p_{\alpha_{2}}),
\]
\[
S_{\alpha}=S_{\alpha,\alpha}.
\]

For properties of $S_{\alpha_{1},\alpha_{2}}$ see \cite{ral}, Lemma 3.1 and
Proposition 3.2. It is straightforward to check that
\[
S_{\alpha,\iota}=linspan\{ (1 \otimes p_{\alpha})\delta(b^{\ast})(1\otimes
p_{\iota})|\; b\in B_{\alpha}\}.
\]

Note that, from the above definition, an element $(1\otimes p_{\alpha}%
)\delta(b)^{\ast}(1\otimes p_{\iota})$ of $S_{\alpha,\iota}$ is an operator
from $H_{u}\otimes\mathbb{C}\xi_{h}$ to $H_{u}\otimes p_{\alpha}H_{h}$ . Since
$p_{\alpha}H_{h}$ is a $d_{\alpha}^{2}$ dimensional subspace of $H_{h}$ (with
basis $\left\{  u_{ij}^{\alpha\ast}|1\leq i,j\leq d_{\alpha}\right\}  $),
every such operator can be represented as a $d_{\alpha}^{2}\times1$ column
matrix with entries in $B$.

If $v$ is the right regular representation, then $ad(v)$ is a coaction of $G$
on $B\times_{\delta}G$, defined by $ad(v)(x)=v_{23}(x\otimes1)v_{23}^{\ast}$,
for all $x\in B\times_{\delta}G$ (\cite{ral}, Lemma 3.3). We consider the
projection $Q$ of $B\times_{\delta}G$ on $(B\times_{\delta}G)^{ad(v)}$, the
$C^{\ast}$-subalgebra of fixed points for the coaction $ad(v)$:%

\[
Q(z)=(id_{B\times_{\delta}G}\otimes h)(ad(v)(z))=(id_{B\times_{\delta}%
G}\otimes h)(v_{23}(z\otimes1)v_{23}^{\ast}).
\]

In \cite{ral}, Section 2.3 it is noticed that, if $u^{\alpha}$ is a
representation of $G$ on a Hilbert space $H$, the following is a coaction of
$G$ on $B\otimes K(H)$:
\[
\delta_{u^{\alpha}}(b\otimes k)=u^{\alpha}_{23}\delta(b)_{13}(1\otimes k
\otimes1){u_{23}^{\alpha}}^{\ast}.
\]

With $\mathcal{I}_{\alpha}=(B\times_{\delta}G)^{ad(v)}\cap S_{\alpha}$, and
$I_{d_{\alpha}}$ the $d_{\alpha}$ dimensional identity matrix, define the map
$\Psi:(B\otimes\mathcal{M}_{d_{\alpha}})^{\delta_{u^{\alpha}}}\to
\mathcal{I}_{\alpha}$,
\[
\Psi(\Lambda)=[\lambda_{ij}\otimes I_{d_{\alpha}}]_{ij},
\]

for each $\Lambda=[\lambda_{ij}]_{ij}\in(B\otimes\mathcal{M}_{d_{\alpha}%
})^{\delta_{u^{\alpha}}}$ (see \cite{ral}, Section 4).

The following result is a generalization of \cite{p}, Lemma 2.10 to the case
of compact quantum groups. The proof uses the matricial representation of the
elements of $S_{\alpha,\iota}$ discussed above. We will use this result in
Section \ref{spectra}.

\begin{lemma}\label{quantum2.10}
$Q(\overline{S_{\alpha , \iota}S_{\iota , \alpha}})=\Psi (\overline{B_{2}^{\delta}(u^{\alpha})^{\ast}B_{2}^{\delta}(u^{\alpha})})$
\end{lemma}

\begin{proof}
Let $b\in B$ and $b_{ij}=P_{ij}^{\alpha}(b)$. Then, if $\eta\in H_{u}$ and
$\xi_{h}$ is the cyclic vector in $H_{h}$, for $i_{0},j_{0}=1,2,...,d_{\alpha
}$ we have:
\begin{align}
(1\otimes p_{\alpha})\delta(b_{i_{0}j_{0}}^{\ast})(1\otimes p_{\iota}%
)(\eta\otimes\xi_{h})  &  = (1\otimes p_{\alpha})(\underset{l=1}%
{\overset{d_{\alpha}}{\sum}}b_{i_{0}l}^{\ast}\otimes{u^{\alpha}_{lj_{0}}%
}^{\ast})(1\otimes p_{\iota})(\eta\otimes\xi_{h})\nonumber\\
&  =\underset{l=1}{\overset{d_{\alpha}}{\sum}}(b_{i_{0}l}^{\ast}%
\otimes(p_{\alpha}{u^{\alpha}_{lj_{0}}}^{\ast}p_{\iota}))(\eta\otimes\xi
_{h})\nonumber\\
&  =\underset{l=1}{\overset{d_{\alpha}}{\sum}}(b_{i_{0}l}^{\ast}%
\otimes\mathcal{F}_{v}(a_{\alpha})^{\ast}{u^{\alpha}_{lj_{0}}}^{\ast
}\mathcal{F}_{v}(1)^{\ast})(\eta\otimes\xi_{h})\nonumber\\
&  =\underset{l=1}{\overset{d_{\alpha}}{\sum}}(b_{i_{0}l}^{\ast}%
\otimes(a_{\alpha}^{\ast}h\ast{u^{\alpha}_{lj_{0}}}^{\ast}))(\eta\otimes
\xi_{h})\nonumber\\
&  =\underset{l,r,m,n}{\sum}(b_{i_{0}l}^{\ast}\otimes M_{\alpha}%
f_{1}(u^{\alpha}_{nm}){u^{\alpha}_{lr}}^{\ast}h({u^{\alpha}_{rj_{0}}}^{\ast
}u^{\alpha}_{mn})(\eta\otimes\xi_{h})\nonumber\\
&  =\underset{l,r,m,n}{\sum} M_{\alpha}f_{1}(u^{\alpha}_{nm})h({u^{\alpha
}_{rj_{0}}}^{\ast}u^{\alpha}_{mn})(b_{i_{0}l}^{\ast}\otimes{u^{\alpha}_{lr}%
}^{\ast})(\eta\otimes\xi_{h})\nonumber\\
&  =\underset{l,r,m,n}{\sum}\delta_{j_{0}n}f_{1}(u^{\alpha}_{nm}%
)f_{-1}(u^{\alpha}_{mr})(b_{i_{0}l}^{\ast}\otimes{u^{\alpha}_{lr}}^{\ast
})(\eta\otimes\xi_{h})\nonumber\\
&  =\underset{l=1}{\overset{d_{\alpha}}{\sum}}b_{i_{0}l}^{\ast}\eta
\otimes{u^{\alpha}_{lj_{0}}}^{\ast}\xi_{h}\nonumber
\end{align}

Consequently, the matrix of $(1\otimes p_{\alpha})\delta(b_{i_{0}j_{0}}^{\ast
})(1\otimes p_{\iota})$ viewed as an operator from $H_{u}\otimes\mathbb{C}%
\xi_{h}$ to $H_{u}\otimes p_{\alpha}H_{h}$ is the $d_{\alpha}^{2}\times1$
column matrix whose entry $[(k-1)d_{\alpha}+j_{0}]\times1$ is $b_{i_{0}%
k}^{\ast}$ and all the other entries are 0. Let now $c\in B$ and
$c_{r_{0}s_{0}}=P_{r_{0}s_{0}}^{\alpha}(c)$. Then, similarly, $(1\otimes
p_{\alpha})\delta(c_{r_{0}s_{0}})(1\otimes p_{\iota})$ can be represented by a
$1\times d_{\alpha}^{2}$ row matrix whose entry $1\times[(k-1)d_{\alpha}%
+s_{0}] $ is $c_{r_{0}k}$ and all the other entries are 0.

Therefore, the product $(1\otimes p_{\alpha})\delta(b_{i_{0}j_{0}}^{\ast
})(1\otimes p_{\iota})\delta(c_{r_{0}s_{0}})(1\otimes p_{\alpha})$ is
represented by a $d_{\alpha}^{2}\times d_{\alpha}^{2}$ matrix $X$, partitioned
in $d_{\alpha}^{2}$ blocks $X_{ij}$, where each block $X_{ij}$ has the entry
$j_{0}s_{0}$ equal to $b_{i_{0}i}^{\ast}c_{r_{0}j}$ and the rest equal to 0,
i.e. $X_{ij}=b_{i_{0}i}^{\ast}c_{r_{0}j}\otimes m_{j_{0}s_{0}}.$

Hence $X=\underset{i,j}{\sum}b_{i_{0}i}^{\ast}c_{r_{0}j}\otimes m_{ij}\otimes
m_{j_{0}s_{0}}.$

On the other hand, by (\cite{ralunitary}, proof of Proposition 9),
$v(p_{\alpha}\otimes1)=\sum I_{d_{\alpha}}\otimes u^{\alpha}=\underset
{p,q}{\sum}I_{d_{\alpha}}\otimes m_{pq}\otimes u^{\alpha}_{pq}$. Hence:
\begin{align}
v_{23}(X\otimes1)v_{23}^{\ast}  &  =\underset{i,j,p,q,k,l}{\sum}b_{i_{0}%
i}^{\ast}c_{r_{0}j}\otimes m_{ij}\otimes m_{pq}m_{j_{0}s_{0}}m_{kl}\otimes
u^{\alpha}_{pq}{u^{\alpha}_{kl}}^{\ast}\nonumber\\
&  =\underset{i,j,p,q,k,l}{\sum}b_{i_{0}i}^{\ast}c_{r_{0}j}\otimes
m_{ij}\otimes\delta_{qj_{0}}\delta_{s_{0}l}m_{pk}\otimes u^{\alpha}%
_{pq}{u^{\alpha}_{kl}}^{\ast}\nonumber\\
&  =\underset{i,j,p,k}{\sum}b_{i_{0}i}^{\ast}c_{r_{0}j}\otimes m_{ij}\otimes
m_{pk}\otimes u^{\alpha}_{pj_{0}}{u^{\alpha}_{ks_{0}}}^{\ast}\nonumber
\end{align}
Applying $id\otimes h$ to the above expression and using Formula \ref{diag2}
above, we get:%

\begin{align}
Q(X)  &  =(\underset{i,j,p,k}{\sum}b_{i_{0}i}^{\ast}c_{r_{0}j}\otimes
m_{ij}\otimes m_{pk})h(u^{\alpha}_{pj_{0}}{u^{\alpha}_{ks_{0}}}^{\ast
})\nonumber\\
&  =\frac{1}{M_{\alpha}}f_{1}(u_{j_{0}s_{0}})\underset{i,j,p,k}{\sum}%
\delta_{pk}b_{i_{0}i}^{\ast}c_{r_{0}j}\otimes m_{ij}\otimes m_{pk}\nonumber\\
&  =\frac{1}{M_{\alpha}}f_{j_{0}}^{\alpha}\delta_{j_{0}s_{0}}\underset
{i,j,p}{\sum}b_{i_{0}i}^{\ast}c_{r_{0}j}\otimes m_{ij}\otimes m_{pp}%
\nonumber\\
&  =\frac{1}{M_{\alpha}}f_{j_{0}}^{\alpha}\delta_{j_{0}s_{0}}\underset
{i,j}{\sum}b_{i_{0}i}^{\ast}c_{r_{0}j}\otimes m_{ij}\otimes I_{d_{\alpha}%
}\nonumber
\end{align}
Hence, if $j_{0}=s_{0}$, we have:
\[
Q(X)=c\underset{i,j}{\sum}b_{i_{0}i}^{\ast}c_{r_{0}j}\otimes m_{ij}\otimes
I_{d_{\alpha}},
\]
where $c=\frac{f_{j_{0}}^{\alpha}}{M_{\alpha}}>0$.

But this is exactly $\Psi(M^{\ast}N)$ where $M\in B_{2}^{\delta}(u^{\alpha})$
is the matrix whose $j_{0}$ row is $[cb_{i_{0}i}]$ and the other entries are 0
and $N\in B_{2}^{\delta}(u^{\alpha})$ is the matrix whose $s_{0}=j_{0}$ row is
$[c_{r_{0}j}]$ and the other entries 0. If $j_{0}\ne s_{0}$ then $Q(X)=0$ but,
as can be easily checked, also $M^{\ast}N=0$ and $\Psi(M^{\ast}N)=0$.
\end{proof}

Let now $(B, G, \delta)$ be a quantum dynamical system. We say that a
$C^{\ast}$-subalgebra $C\subset B$ is $\delta$-invariant if the following two
conditions hold:

\begin{itemize}
\item[(1)] $\delta(C)\subseteq C\otimes A$\notag

\item[(2)] $\overline{\delta(C)(1\otimes A)}=C\otimes A$\notag

\end{itemize}

In other words, $C$ is called $\delta$-invariant if the restriction of
$\delta$ to $C$ is a coaction. The set of all hereditary, $\delta$-invariant
$C^{\ast}$-subalgebras of $B$ will be denoted by $\mathcal{H}^{\delta}(B)$.

A C*-algebra $B$ is called $G$-prime if the product of two non-zero $\delta
$-invariant ideals is non-zero.

A C*-algebra $B$ is called $G$-simple if $B$ has no non-trivial $\delta
$-invariant two sided ideals.

We will need the following remarks. Their proofs are straightforward.

\begin{remark}\label{rem1}
\begin{itemize}
\item[(1)]$S_{\iota}=B^{\delta}\otimes 1$
\item[(2)]Using the proof of Proposition 3.2 in \cite{ral}, one can show that for $a_{0},a_{1}\in B^{\delta}$ and $\alpha \in \widehat{G}$, then $a_{1}B_{\alpha}^{\delta}a_{0}=(0)$ if and only if $(a_{0}\otimes 1)S_{\alpha,\iota}(a_{1}\otimes 1)=(0)$.
\item[(3)] If $C\in \mathcal{H}^{\delta}(B)$, then $C\times_{\delta}G$ is a hereditary subalgebra of $B\times_{\delta}G$.
\item[(4)] If $J\subset B^{\delta}$ is a two sided ideal, then $D=\overline{JBJ}\in \mathcal{H}^{\delta}(B)$.
\end{itemize}
\end{remark}

The next lemma and its corollary will be used in Section \ref{strong}.

\begin{lemma}\label{cc}
Let $\alpha \in \widehat{G}$. Then $S_{\alpha}=\overline{linspan\{S_{\alpha,\beta}S_{\beta,\alpha}|\beta \in \widehat{G}\}}^{\|\cdot \|}$.
\end{lemma}

\begin{proof}
Since $\underset{\beta\in\widehat{G}}{\sum}p_{\beta}=1$ in the strict topology
of $\widehat{A}$ we have
\[
(1\otimes p_{\alpha})(B\times_{\delta}G)(1\otimes p_{\alpha})=(1\otimes
p_{\alpha})(B\times_{\delta}G)\underset{\beta\in\widehat{G}}{\sum}(1\otimes
p_{\beta})(B\times_{\delta}G)(1\otimes p_{\alpha})
\]
and the claim follows.
\end{proof}

\begin{corollary}\label{cor}
Let $J\subset B^{\delta}$ be a two sided ideal. Then $C=\overline{BJB}\in \mathcal{H}^{\delta}(B)$ and $$C^{\delta}\otimes 1 = \overline{linspan\{S_{\iota,\beta }(J\otimes 1) S_{\beta, \iota}| \beta \in \widehat{G}\}}^{\| \cdot \|}$$
\end{corollary}

\begin{proof}
Clearly, $\delta(C) \subseteq C\otimes A$, since $\delta(J)=J\otimes1$. The
fact that $\delta(C)(1\otimes A)$ is dense in $C\otimes A$ follows from the
definition of $C$.

The equality $C^{\delta} \otimes1 = \overline{linspan\{S_{\iota,\beta
}(J\otimes1) S_{\beta, \iota}| \beta\in\widehat{G}\}}^{\| \cdot\|}$ follows
from Lemma \ref{cc} and Remark \ref{rem1} (1).
\end{proof}

\section{Connes Spectrum and prime crossed products}

\label{spectra}

A notion of spectrum of an action $\delta$ of a compact group $G$ on a
$C^{\ast}$-algebra $B$ was given in \cite{glp} by Gootman, Lazar, and
Peligrad. They used the spectral subspaces $B_{2}^{\delta}(\alpha)$ to define
the Arveson and Connes spectra and proved that the conjugate $\overline
{\alpha}$ belongs to the Arveson spectrum $Sp(\delta)$ if and only if the
closure of the ideal $S_{\alpha,\iota}*S_{\iota,\alpha}$ is essential in
$S_{\alpha}$ (Proposition 1.3). We are going to use this correspondence rather
than the direct definition given in \cite{glp} to define the spectra for
coactions of a compact quantum group on a $C^{\ast}$-algebra $B$.

\begin{definition}\label{arvconnesspectra}
(1) $Sp(\delta)=\{\alpha\in \widehat{G}|S_{\overline{\alpha},\iota}S_{\iota,\overline{\alpha}}$ is an essential ideal of $S_{\overline{\alpha}}\}$
(2) $\Gamma(\delta)=\underset{C\in\mathcal{H}^{\delta}(B)}{\cap} Sp(\delta|_{C})$
\end{definition}

The connection to the definition in \cite{glp} is made by the following lemma.

\begin{lemma}\label{essential}Let $\alpha \in \widehat{G}$. Then $\alpha \in Sp(\delta)$ if and only if $\overline{B_{2}^{\delta}(u^{\alpha})^{\ast}B_{2}^{\delta}(u^{\alpha})}$ is an essential ideal of $(B\otimes \mathcal{M}_{d_{\alpha}}(\mathbf{C}))^{\delta_{u^{\alpha}}}$.
\end{lemma}

\begin{proof}
Let $\alpha\in Sp(\delta)$ and assume to the contrary that $\overline
{B_{2}^{\delta}(u^{\alpha})^{\ast}B_{2}^{\delta}(u^{\alpha})}$ is not an
essential ideal of $(B\otimes\mathcal{M}_{d_{\alpha}}(\mathbf{C}%
))^{\delta_{u^{\alpha}}}$.

Using Lemma \ref{quantum2.10}, there exists a positive, non-zero element
$c\in\mathcal{I}_{\overline{\alpha}}$, such that $c\mathcal{P}(S_{\overline
{\alpha},\iota}S_{\iota, \overline{\alpha}})c=0$. Since , in particular, $c\in
S_{\overline{\alpha}}$ then $\mathcal{P}(c(S_{\overline{\alpha},\iota}%
S_{\iota, \overline{\alpha}})c)=0$. The faithfulness of $\mathcal{P}$ implies
now that $c(S_{\overline{\alpha},\iota}S_{\iota, \overline{\alpha}})c=0$,
which is a contradiction with $\alpha\in Sp(\delta)$.

Conversely, assume that $\overline{B_{2}^{\delta}(u^{\alpha})^{\ast}%
B_{2}^{\delta}(u^{\alpha})}$ is an essential ideal of $(B\otimes
\mathcal{M}_{d_{\alpha}}(\mathbf{C}))^{\delta_{u^{\alpha}}}$. By Lemma
\ref{quantum2.10}, $\mathcal{P}(S_{\overline{\alpha},\iota}S_{\iota,
\overline{\alpha}})$ is an essential ideal in $\mathcal{I}_{\overline{\alpha}%
}$. Using the same Lemma,%

\[
\mathcal{P}(\overline{S_{\overline{\alpha},\iota}S_{\iota, \overline{\alpha}}%
})=\mathcal{I}_{\overline{\alpha}}\cap(\overline{S_{\overline{\alpha},\iota
}S_{\iota, \overline{\alpha}}}).
\]

By Remark 3.5 in \cite{ral}, $S_{\overline{\alpha}}$ is isomorphic to
$I({\overline{\alpha}})\otimes\mathcal{I}_{\overline{\alpha}}$, where
$I({\overline{\alpha}})=\widehat{A}p_{\overline{\alpha}}$. It is easy to check
that the image of $\mathcal{I}_{\overline{\alpha}}\subset S_{\overline{\alpha
}}$ by this isomorphism is $\chi_{\overline{\alpha}}\otimes
\{diag(x,x,...,x)|x\in\mathcal{I}_{\overline{\alpha}}\}$, where
$diag(x,x,...,x)$ is the $d_{\alpha}\times d_{\alpha}$ matrix with all the
diagonal elements equal to $x$ and all the others equal to $0$. Thus
$\{diag(y,y,...,y)|y\in\mathcal{I}_{\overline{\alpha}}\cap\overline
{S_{\overline{\alpha},\iota}S_{\iota, \overline{\alpha}}}\}$ is essential in
$\{diag(x,x,...,x)|x\in\mathcal{I}_{\overline{\alpha}}\}$. This implies that
$\overline{S_{\overline{\alpha},\iota}S_{\iota, \overline{\alpha}}}$ is
essential in $S_{\overline{\alpha}}$.
\end{proof}

\begin{proposition}\label{prop}If $B$ is $G$-prime and $\Gamma(\delta)=\widehat{G}$, then $B^{\delta}$ is prime.
\end{proposition}

\begin{proof}
Assume, to the contrary, that $B^{\delta}$ is not prime. Then there exist two
non-zero positive elements $a_{0},a_{1}\in B^{\delta}$ such that
$a_{1}B^{\delta}a_{0}=(0)$. Since $B$ is $G$-prime, $a_{1}Ba_{0}\ne(0)$. On
the other hand, since $B$ is the closure of the linear span of its spectral
subspaces $\{B^{\delta}_{\alpha}|\alpha\in\widehat{G}\}$, then there exists
$\alpha_{0}\in\widehat{G}$ such that
\begin{equation}
\label{0}a_{1}B^{\delta}_{\alpha_{0}}a_{0}\ne(0)
\end{equation}

Since $a_{1}B^{\delta}a_{0}=(0)$, Remark \ref{rem1}(1) above implies that%

\[
(1\otimes p_{\iota})((a_{1}\otimes1)(B\times_{\delta}G))(1\otimes p_{\alpha
})(B\times_{\delta}G)(a_{0}\otimes1)(1\otimes p_{\iota})=(0)
\]
that is%

\[
(a_{1}\otimes1)S_{\iota,\overline{\alpha_{0}}}S_{\overline{\alpha_{0}},\iota
}(a_{0}\otimes1)=(0)
\]
Therefore, since $S_{\iota,\overline{\alpha_{0}}}(a_{0}^{2}\otimes
1)S_{\overline{\alpha_{0}},\iota}\subset S_{\iota,\overline{\alpha_{0}}%
}S_{\overline{\alpha_{0}},\iota}$, then%

\begin{equation}
(a_{1}\otimes1)S_{\iota,\overline{\alpha_{0}}}(a_{0}^{2}\otimes1)S_{\overline
{\alpha_{0}},\iota}(a_{0}\otimes1)=(0)
\end{equation}

Multiply the above equation to the left by $(a_{0}\otimes1)S_{\overline
{\alpha_{0}},\iota}(a_{1}\otimes1)$ and to the right by $(a_{0}\otimes
1)S_{\iota,\overline{\alpha_{0}}}(a_{0}\otimes1)$. We get:%

\begin{equation}
\label{aaa}(a_{0}\otimes1)S_{\overline{\alpha_{0}},\iota}(a_{1}^{2}%
\otimes1)S_{\iota,\overline{\alpha_{0}}}(a_{0}^{2}\otimes1)S_{\overline
{\alpha_{0}},\iota}(a_{0}^{2}\otimes1)S_{\iota,\overline{\alpha_{0}}}%
(a_{0}\otimes1)=(0)
\end{equation}

Regroup the terms in the left-hand side of equation \ref{aaa} as:%

\begin{equation}
\label{bbb}[(a_{0}\otimes1)S_{\overline{\alpha_{0}},\iota}(a_{1}^{2}%
\otimes1)S_{\iota,\overline{\alpha_{0}}}(a_{0}\otimes1)][(a_{0}\otimes
1)S_{\overline{\alpha_{0}},\iota}(a_{0}^{2}\otimes1)S_{\iota,\overline
{\alpha_{0}}}(a_{0}\otimes1)]=(0)
\end{equation}

Let $C=\overline{a_{0}Ba_{0}}$. Clearly, $C\in\mathcal{H}^{\delta}(B)$. The
second factor in the left-hand side of equation \ref{bbb} is just
$S^{c}_{\overline{\alpha_{0}},\iota}S^{c}_{\iota,\overline{\alpha_{0}}}$ where
$S^{c}_{\alpha,\beta}$ denotes the corresponding subspace of the crossed
product $C\times_{\delta}G$.

Since $\Gamma(\delta)=\widehat{G}$, then $\overline{\alpha_{0}}\in
\Gamma(\delta)$. Therefore $S^{c}_{\overline{\alpha_{0}},\iota}S^{c}%
_{\iota,\overline{\alpha_{0}}}$ is an essential ideal of $S^{c}_{\overline
{\alpha_{0}}}$. Since the first factor in the equation \ref{bbb} is included
in $S^{c}_{\overline{\alpha_{0}}}$, it follows that it equals (0). In
particular,
\[
(a_{0}\otimes1)S_{\overline{\alpha_{0}},\iota}(a_{1}\otimes1)=(0).
\]
Using Remark \ref{rem1} (2), this means that $a_{1}B^{\delta}_{\alpha_{0}%
}a_{0}=(0)$, which is a contradiction with relation \ref{0}.
\end{proof}

We will prove next the main result of this section. The result is a
generalization of \cite{glp}, Theorem 2.2.

\begin{theorem}\label{primeness}
The following are equivalent:
\begin{itemize}
\item[(1)]$B\times_{\delta}G$ is prime
\item[(2)]$B$ is $G$-prime and $\Gamma(\delta)=\widehat{G}$.
\end{itemize}
\end{theorem}

\begin{proof}
Assume that $B\times_{\delta}G$ is prime. Since for every $\delta$-invariant
ideal $J\subset B$, $J\times_{\delta}G$ is an ideal of $B\times_{\delta}G$,
the fact that $B$ is $G$-prime is immediate. We will show next that
$\Gamma(\delta)=\widehat{G}$.

Let $C\in\mathcal{H}^{\delta}(B)$ and $\alpha\in\widehat{G}$. By Remark
\ref{rem1} (3) above, $C\times_{\delta}G$ is a hereditary subalgebra of
$B\times_{\delta}G$ and is therefore prime. Using \cite{ral} Corollary 4.9,
$(C\otimes\mathcal{M}_{d_{\alpha}})^{\delta_{u^{\alpha}}}$ is prime and
$C_{2}^{\delta}(\alpha)\ne(0)$ (since $C^{\delta}_{\alpha}\ne(0)$). Thus the
ideal $\overline{C_{2}^{\delta}(\alpha)^{\ast}C_{2}^{\delta}(\alpha)}$ is
essential in $(C\otimes\mathcal{M}_{d_{\alpha}})^{\delta_{u^{\alpha}}}$.
Therefore $\alpha\in\Gamma(\delta)$ and so $\Gamma(\delta)=\widehat{G}$.

Conversely, assume that $B$ is $G$-prime and $\Gamma(\delta)=\widehat{G}$.

For each $\alpha\in\widehat{G}$, the $C^{\ast}$-algebras $\overline
{S_{\alpha,\iota}S_{\iota, \alpha}}$ and $\overline{S_{\iota, \alpha}%
S_{\alpha, \iota}}$ are strongly Morita equivalent ($S_{\alpha, \iota}$ being
the imprimitivity bimodule). By Proposition \ref{prop}, $B^{\delta}$ is prime
and therefore, by Remark \ref{rem1} (1), $S_{\iota}$ is prime. Since
$S_{\iota}$ is prime, so is the ideal $\overline{S_{\iota, \alpha}S_{\alpha,
\iota}}$ and the Morita equivalent algebra $\overline{S_{\alpha,\iota}%
S_{\iota, \alpha}}$. By the definition of $\Gamma(\delta)$, $\overline
{S_{\alpha,\iota}S_{\iota, \alpha}}$ is an essential ideal of $S_{\alpha}$ and
thus $S_{\alpha}$ is prime also. The implication follows now from \cite{ral},
Corollary 4.9.
\end{proof}

\section{Strong Connes spectrum and simple crossed products}

\label{strong}

We begin by defining the strong Arveson and Connes spectra for compact quantum
group coactions.

\begin{definition}\label{strongspectra}\hfill
\begin{itemize}
\item[(1)]Strong Arveson Spectrum
$\tilde{Sp}(\delta)=\{\alpha \in \widehat{G}|\overline{S_{\overline{\alpha},\iota}S_{\iota,\overline{\alpha}}}^{\|\cdot \|}=S_{\overline{\alpha}}\}$.
\item[(2)] Strong Connes spectrum
$\tilde{\Gamma}(\delta)=\underset{c\in\mathcal{H}^{\alpha}(B)}{\cap}\tilde{Sp}(\delta|_{C})$.
\end{itemize}
\end{definition}

Using similar arguments as in Lemma \ref{essential} we obtain the following result.

\begin{lemma}
Let $\alpha \in \widehat{G}$. Then $\alpha \in \tilde{Sp}(\delta)$ if and only if $\overline{B^{\delta}_{2}(u^{\alpha})^{\ast}B^{\delta}_{2}(u^{\alpha})}=(B\otimes \mathcal{M}_{d_{\alpha}})^{\delta_{u^{\alpha}}}$
\end{lemma}

The following result makes a connection between the strong Connes spectrum and
the simplicity of the fixed point algebra $B^{\delta}$.

\begin{proposition}\label{propsimple}
If $B$ is $G$-simple and $\tilde{\Gamma}(\delta)=\widehat{G}$, then $B^{\delta}$ is simple.
\end{proposition}

\begin{proof}
Let $J\subseteq B^{\delta}$ be a non-zero two sided ideal. We will prove that
$J=B^{\delta}$ and thus $B^{\delta}$ is simple. To this end we will show that
$S_{\iota,\alpha}(J\otimes1) S_{\alpha,\iota}\subseteq J\otimes1$, for any
$\alpha\in\widehat{G}$. The claim will then follow from Corollary \ref{cor}.

Let $D=\overline{JBJ}$ and let $\alpha\in\widehat{G}$. Then $D\in
\mathcal{H}^{\delta}(B)$. Since $\alpha\in\widehat{G}=\tilde{\Gamma}(\delta)$,
we have%

\[
\overline{S^{D}_{\alpha, \iota}S^{D}_{\iota, \alpha}}=S^{D}_{\alpha}%
\]

where $S^{D}_{\alpha,\iota}$ and $S^{D}_{\iota,\alpha}$ are the corresponding
subspaces of $D\times_{\delta}G$.

By the definition of $D$, it obviously follows that $S^{D}_{\alpha,\iota
}=\overline{(J\otimes1)S_{\alpha,\iota}(J\otimes1)}$ and $S^{D}_{\alpha
}=\overline{(J\otimes1)S_{\alpha}(J\otimes1)}$.

Therefore,%

\begin{equation}
\label{ccc}\overline{(J\otimes1)S_{\alpha,\iota}(J\otimes1)S_{\iota,\alpha
}(J\otimes1)}=\overline{(J\otimes1)S_{\alpha}(J\otimes1)}%
\end{equation}

Multiplying equation \ref{ccc} to the left by $S_{\iota,\alpha}$ and to the
right by $S_{\alpha,\iota}$ we get%

\begin{equation}
\label{ddd}\overline{S_{\iota,\alpha}(J\otimes1)S_{\alpha,\iota}%
(J\otimes1)S_{\iota,\alpha}(J\otimes1)S_{\alpha,\iota}}=\overline
{S_{\iota,\alpha}(J\otimes1)S_{\alpha}(J\otimes1)S_{\alpha,\iota}}%
\end{equation}

By Remark \ref{rem1} (1), $S_{\iota,\alpha}S_{\alpha,\iota}\subseteq S_{\iota
}=B^{\delta}\otimes1$ and since $J\subseteq B^{\delta}$, the left-hand side of
equation \ref{ddd} is included in $J\otimes1$. Therefore, the right-hand side
of equation \ref{ddd} is included in $J\otimes1$:%

\begin{equation}
\label{eee}S_{\iota,\alpha}(J\otimes1)S_{\alpha}(J\otimes1)S_{\alpha,\iota
}\subseteq J \otimes1
\end{equation}

Since $S_{\iota,\alpha}(J\otimes1)S_{\alpha,\iota}\subseteq S_{\iota,\alpha
}(J\otimes1)S_{\alpha}(J\otimes1)S_{\alpha,\iota}$, from equation \ref{eee} it follows:%

\begin{equation}
S_{\iota,\alpha}(J\otimes1)S_{\alpha,\iota}\subseteq J\otimes1
\end{equation}
and we are done.
\end{proof}

We can now prove:

\begin{theorem}\label{simple}
The following are equivalent:
\begin{itemize}
\item[(1)] $B\times_{\delta}G$ is simple.
\item[(2)] $B$ is $G$-simple and $\tilde{\Gamma}(\delta)=\widehat{G}$.
\end{itemize}
\end{theorem}

\begin{proof}

Assume first that $B\times_{\delta}G$ is simple. That $B$ is $G$-simple
follows easily since for every non-trivial ideal $J\in\mathcal{H}^{\delta}%
(B)$, $J\times_{\delta}G$ is a non-trivial ideal of $B\times_{\delta}G$.

Let now $\alpha\in\widehat{G}$ and $C\in\mathcal{H}^{\delta}(B)$. Then
$C\times_{\delta}G$ si a hereditary subalgebra of $B\times_{\delta}G$ by
Remark \ref{rem1} (3) and hence it is simple. By \cite{ral} Corollary 4.9,
$S_{\alpha,\iota}\ne0$ and $S_{\alpha}$ is simple. Hence $\overline
{S_{\alpha,\iota}S_{\iota,\alpha}}=S_{\alpha}$ and $\alpha\in\tilde{\Gamma
}(\delta)$.

Conversely, assume that $B$ is $G$-simple and $\tilde{\Gamma}(\delta
)=\widehat{G}$. By Proposition \ref{propsimple}, $B^{\delta}$ is simple. Now,
for every $\alpha\in\widehat{G}=\tilde{\Gamma}(\delta)$, the non-zero ideal
$\overline{S_{\alpha,\iota}S_{\iota,\alpha}}\subseteq S_{\iota}=B^{\delta
}\otimes1$ is simple and so is the Morita equivalent algebra $\overline
{S_{\alpha,\iota}S_{\iota,\alpha}}=S_{\alpha}$. The conclusion follows now
from \cite{ral}, Corollary 4.9.
\end{proof}

\section{ Spectra are closed under tensor products}

In order to prove the results about the stability of the Connes spectrum and
the strong Connes spectrum to tensor products, we need to make some notations.
If $\alpha\in\widehat{G}$ and $\beta\in\widehat{G}$ and $u^{\alpha}\in\alpha$,
$u^{\beta}\in\beta$ denote by $u^{\alpha}\odot u^{\beta}=\sum_{p,q,r,s}%
m_{pq}^{\alpha}\otimes m_{rs}^{\beta}\otimes u_{pq}^{\alpha}u_{rs}^{\beta}$
the Kronecker tensor product of $u^{\alpha} $ and $u^{\beta}$, which is a
representation of $A$ \cite{wor2}. Then $u^{\alpha}\odot u^{\beta}$ is unitary
if both $u^{\alpha}$and $u^{\beta}$ are unitary. Moreover, $u^{\alpha}\odot
u^{\beta}$ is equivalent to a direct sum of irreducible representations,
$u^{\alpha}\odot u^{\beta}\widetilde{=}\sum_{i}^{\oplus}u^{\rho_{i}},\rho
_{i}\in\widehat{G}$. The equivalence and $\rho_{i}\in\widehat{G}$ are unitary
if both $u^{\alpha}$and $u^{\beta}$ are unitary \cite{wor2}.

\begin{definition}
Let $\Pi\subset\widehat{G}$ be a subset. We say that $\Pi$ is closed under tensor products if for every $\alpha\in\Pi,\beta\in\Pi$ and $u^{\alpha}\in\alpha$, $u^{\beta}\in\beta$ it follows that every irreducible component of $u^{\alpha}\odot u^{\beta}$ belongs to $\Pi$.
\end{definition}

If $X\in B_{2}^{\delta}(u^{\alpha})$ and $Y\in B_{2}^{\delta}(u^{\beta})$ we
denote $X\odot Y=\sum_{l,k,i,j}X_{lk}Y_{ij}\otimes m_{lk}^{\alpha}\otimes
m_{ij}^{\beta}$ (for the case of groups this notation was used in
\cite{peligrad}). Standard calculations show that $X\odot Y\in B_{2}^{\delta
}(u^{\alpha}\odot u^{\beta})$. Furthermore, $X\odot Y$ can be viewed as the
matrix of order $d_{\alpha}d_{\beta}\times d_{\alpha}d_{\beta}$ partitioned in
$d_{\beta}^{2}$ blocks of order $d_{\alpha}\times d_{\alpha}$ as follows:
$X\odot Y=[Xdiag(Y_{ij})]$, where $diag(Y_{ij})$ is the $d_{\alpha}\times
d_{\alpha}$ matrix with all the diagonal entries equal to $Y_{ij}$ and all the
others equal to $0$.

\begin{remark}\label{odot}
If $u^{\alpha}\odot u^{\beta}\widetilde{=}\sum_{i}^{\oplus}u^{\rho_{i}},\rho_{i}\in\widehat{G}$, then:
(1)$(B\otimes\mathcal{M}_{d_{\alpha}d_{\beta}})^{\delta_{u^{\alpha}\odot u^{\beta}}}$ is spatially isomorphic to $\sum^{\oplus}(B\otimes \mathcal{M}_{d_{\rho_{i}}})^{\delta_{\rho_{i}}}$ \ ($\ast$-isomorphic if both $u^{\alpha}$and $u^{\beta}$ are unitary) and
(2)$B_{2}^{\delta}(u^{\alpha}\odot u^{\beta})$ is spatially isomorphic to $\sum^{\oplus}B_{2}^{\delta}(\rho_{i})$.
\end{remark}

The proof of the above remark follows immediately using a change of basis in
$\mathcal{M}_{d_{\alpha}d_{\beta}}$ .

We prove first

\begin{lemma}
$\tilde{Sp}(\delta|_{C})$ is closed under tensor products for every $C\in\mathcal{H}^{\delta}(B)$.
\end{lemma}

\begin{proof}
We have to prove that if $\alpha,\beta\in\tilde{Sp}(\delta|C),C\in
\mathcal{H}^{\delta}(B)$, then every irreducible component of $u^{\alpha}\odot
u^{\beta}$ belongs to $\tilde{Sp}(\delta|C)$.

It is enough to prove the above claim for $C=B$. We first show that if
$\alpha\in\tilde{Sp}(\delta)$ and $\beta\in\tilde{Sp}(\delta)$, then
$B_{2}^{\delta}(u^{\alpha}\odot u^{\beta})^{\ast}B_{2}^{\delta}(u^{\alpha
}\odot u^{\beta})$ is a dense ideal of $(B\otimes\mathcal{M}_{d_{\alpha
}d_{\beta}})^{\delta_{u^{\alpha}\odot u^{\beta}}}$.

Indeed, by (\cite{brown}, Theorem 2.1), $(B\otimes\mathcal{M}_{d_{\alpha}%
})^{\delta_{\alpha}}$ has an approximate identity $\{E_{\lambda}\}$ of the
form $E_{\lambda}=\sum_{1}^{n_{\lambda}}(X_{i}^{\lambda})^{\ast}X_{i}%
^{\lambda}$, $X_{i}^{\lambda}\in B_{2}^{\delta}(u^{\alpha}),i=1,2...n_{\lambda
}$. By (\cite{ralpel}, Lemma 2.7) $\{E_{\lambda}\}$ is an approximate identity
of $B\otimes\mathcal{M}_{d_{\alpha}}$. Hence $(Y_{1}\odot I_{d_{\alpha}%
})^{\ast}(Y_{2}\odot I_{d_{\alpha}})\in B_{2}^{\delta}(u^{\alpha}\odot
u^{\beta})$, for all $Y_{1},Y_{2}\in B_{2}^{\delta}(u^{\beta})$. Since
$\beta\in\tilde{Sp}(\delta)$, $B_{2}^{\delta}(u^{\beta})^{\ast}B_{2}^{\delta
}(u^{\beta})$ is a dense ideal of $(B\otimes\mathcal{M}_{d_{\beta}}%
)^{\delta_{\beta}}$. Using an approximate identity of $(B\otimes
\mathcal{M}_{d_{\beta}})^{\delta_{\beta}}$ of the form $F_{\gamma}=\sum
_{1}^{m_{\gamma}}(Y_{i}^{\gamma})^{\ast}Y_{i}^{\gamma},Y_{i}^{\gamma}\in
B_{2}^{\delta}(u^{\beta})$, by the pattern we used above, it follows that
$B_{2}^{\delta}(u^{\alpha}\odot u^{\beta})^{\ast}B_{2}^{\delta}(u^{\alpha
}\odot u^{\beta})$ is a dense ideal of $(B\otimes\mathcal{M}_{d_{\alpha
}d_{\beta}})^{\delta_{u^{\alpha}\odot u^{\beta}}}$.

On the other hand, since $u^{\alpha}\odot u^{\beta}$ is equivalent to a direct
sum of irreducible representations, $u^{\alpha}\odot u^{\beta}\widetilde
{=}\sum_{i}^{\oplus}u^{\rho_{i}},\rho_{i}\in\widehat{G},$ by Remark
\ref{odot}(1), $(B\otimes\mathcal{M}_{d_{\alpha}d_{\beta}})^{\delta
_{u^{\alpha}\odot u^{\beta}}}$ is spatially $\ast$-isomorphic to $\sum
^{\oplus}(B\otimes\mathcal{M}_{d_{\rho_{i}}})^{\delta_{\rho_{i}}}$. Thus,
since by Remark \ref{odot}(2), $B_{2}^{\delta}(u^{\alpha}\odot u^{\beta})$ is
spatially isomorphic to $\sum^{\oplus}B_{2}^{\delta}(\rho_{i})$, it follows
that $B_{2}^{\delta}(\rho_{i})^{\ast}B_{2}^{\delta}(\rho_{i})$ is dense in
$(B\otimes\mathcal{M}_{d_{\rho_{i}}})^{\delta_{\rho_{i}}}$ for all $i$.
Therefore $\rho_{i}\in\tilde{Sp}(\delta)$ for every $i$. Thus $\tilde
{Sp}(\delta|_{C})$ is closed under tensor products for every $C\in
\mathcal{H}^{\delta}(B)$.
\end{proof}

Therefore:

\begin{proposition}\label{strongconnesclosed}
$\tilde{\Gamma}(\delta)$ is closed under tensor products.
\begin{proof}
Obvious, since $\tilde{\Gamma}(\delta)=\underset{C\in\mathcal{H}^{\alpha}(B)}{\cap}\tilde{Sp}(\delta|_{C}).$
\end{proof}
\end{proposition}

We will prove next that the Connes spectrum is closed under tensor products.
As in the case of the strong Connes spectrum, we will show first that our
Arveson spectrum $Sp(\delta|_{C})$, is closed under tensor products for every
$C\in\mathcal{H}^{\delta}(B)$.

Let $\alpha\in\widehat{G}$ and $\beta\in\widehat{G}$ and $u^{\alpha}\in\alpha
$, $u^{\beta}\in\beta$ . If $u^{\alpha}$ and $u^{\beta}$ are unitary, then, as
noticed above, $u^{\alpha}\odot u^{\beta}$ is a unitary representation. If
$u_{1}^{\alpha}$ is a representation in the class $\alpha$, not necessarily
unitary, then there exists an invertible matrix $S\in\mathcal{M}_{d_{\alpha}}$
such that $u_{1}^{\alpha}=(S^{-1}\otimes1)u^{\alpha}(S\otimes1)$. Notice that
$\ B_{2}^{\delta}(u_{1}^{\alpha})=\{(1_{B}\otimes S^{-1})X(1_{B}\otimes
S)|X\in B_{2}^{\delta}(u^{\alpha})\}$.

\begin{lemma}
$Sp(\delta|_{C})$ is closed under tensor products for every $C\in\mathcal{H}^{\delta}(B)$.
\end{lemma}

\begin{proof}
We may assume that $C=B$. Let $\alpha,\beta\in Sp(\delta)$ and $u^{\alpha}%
\in\alpha,u^{\beta}\in\beta$ be unitary representatives of $\alpha$ and
$\beta$. We will first show that%

\[
linspan\{(X\odot Y)^{\ast}(X\odot Y)|X\in B_{2}^{\delta}(u^{\alpha}),Y\in
B_{2}^{\delta}(u^{\beta})\}
\]
is an essential ideal of $(B\otimes\mathcal{M}_{d_{\alpha}d_{\beta}}%
)^{\delta_{u^{\alpha}\odot u^{\beta}}}$. It then follows immediately that each
irreducible component of $u^{\alpha}\odot u^{\beta}$ belongs to $Sp(\delta)$.

Let $Z\in(B\otimes\mathcal{M}_{d_{\alpha}d_{\beta}})^{\delta_{u^{\alpha}\odot
u^{\beta}}},Z\geq0$. Assume that $(X\odot Y)Z=0$, for every $X\in
B_{2}^{\delta}(u^{\alpha}),Y\in B_{2}^{\delta}(u^{\beta})$. Let $Z$ be
partitioned in blocks as follows: $Z=\sum_{l,k=1}^{d_{\beta}}Z_{lk}\otimes
m_{lk}^{\beta}$, where $Z_{lk}$ are $d_{\alpha}\times d_{\alpha}$ matrices
with entries in $B$. Since $(X\odot Y)Z=0$, for every $X\in B_{2}^{\delta
}(u^{\alpha}), Y\in B_{2}^{\delta}(u^{\beta})$, it follows that $(X\odot
Y)Z(I_{d\alpha}\odot Y^{\ast})=0$, for every such $X,Y$. In particular, if $Y$
is as in Remark \ref{rem2.1} (4), that is $Y$ has only one nonzero row
consisting of $y_{1},y_{2},...y_{d_{\beta}}$, we have%

\[
X\sum_{i,j=1}^{d_{\beta}}y_{i}Z_{ij}y_{j}^{\ast}=0,
\]
for every $X\in B_{2}^{\delta}(u^{\alpha})$ and $Y\in B_{2}^{\delta}(u^{\beta
})$ as chosen, where the multiplication $y_{i}Z_{ij}y_{j}^{\ast}$ is the
multiplication in $B$ of $y_{i},y_{j}^{\ast}$ with each entry of $Z_{ij}$. We
prove first the following:%

\begin{equation}
\label{formula1}\sum y_{i}Z_{ij}y_{j}^{\ast}\in(B\otimes\mathcal{M}%
_{d_{\alpha}})^{\delta_{u^{\alpha}}},
\end{equation}
for every $Y\in B_{2}^{\delta}(u^{\beta})$ as chosen (i.e. with only one
nonzero row).

In the following leg numbering notation, there are four places in the
following order: $B,\mathcal{M}_{d_{\alpha}},\mathcal{M}_{d_{\beta}},A$.

Since $Z\in(B\otimes\mathcal{M}_{d_{\alpha}d_{\beta}})^{\delta_{u^{\alpha
}\odot u^{\beta}}}$, we have:%

\begin{equation}
\label{formula2}\delta_{14}(Z)=(1_{B}\otimes(u^{\alpha}\odot u^{\beta})^{\ast
})(Z\otimes1_{A})(1_{B}\otimes(u^{\alpha}\odot u^{\beta})).
\end{equation}

By the definition of $u^{\alpha}\odot u^{\beta}$, we have:%

\begin{equation}%
\begin{split}
\delta_{14}(\sum Z_{ij}\otimes m_{ij}^{\beta}) &  =(\sum1_{B}\otimes
m_{pq}^{\alpha}\otimes m_{rs}^{\beta}\otimes u_{sr}^{\beta\ast}u_{qp}%
^{\alpha\ast})(\sum Z_{ij}\otimes m_{ij}^{\beta}\otimes1_{A})\times\\
&  (\sum1_{B}\otimes m_{tu}^{\alpha}\otimes m_{vw}^{\beta}\otimes
u_{pq}^{\alpha}u_{rs}^{\beta})
\end{split}
\label{formula3}%
\end{equation}

On the other hand, taking into account that $Y\in(B\otimes\mathcal{M}%
_{d_{\beta}})^{\delta_{u^{\beta}}}$, it follows that:%

\begin{equation}
\label{formula4}\delta_{14}(Y_{13})=(1_{B}\otimes(u^{\beta})_{34}^{\ast
})(Y_{13}\otimes1_{A})(1_{B}\otimes(u^{\beta})_{34}),
\end{equation}
where $(u^{\beta})_{34}=\sum1_{B}\otimes I_{d_{\alpha}}\otimes m_{lk}^{\beta
}\otimes u_{lk}^{\beta}$.

By combining Formulas \ref{formula2} and \ref{formula4} and taking into
account that $u^{\alpha}$ and $u^{\beta}$ are unitary, we get Formula
\ref{formula1}. Therefore, since $\alpha\in Sp(\delta)$ it follows that $\sum
y_{i}Z_{ij}y_{j}^{\ast}=0$ for every such $Y$.

Let $u_{1}^{\alpha}\in\alpha$ be a not necessarily unitary representation, but
such that $\overline{u_{1}^{\alpha}}$ is unitary. Then, since $u^{\alpha}$ and
$u_{1}^{\alpha}$ are equivalent, there is an invertible matrix $S\in
\mathcal{M}_{d_{\alpha}}$ such that $u_{1}^{\alpha}=(S^{-1}\otimes1)u^{\alpha
}(S\otimes1)$. Notice that $B_{2}^{\delta}(u_{1}^{\alpha})=\{(1_{B}\otimes
S^{-1})X(1_{B}\otimes S)|X\in B_{2}^{\delta}(u^{\alpha})\}$. Denote
$V_{ij}=(1_{B}\otimes S^{\ast})Z_{ij}(1_{B}\otimes S)$ for all
$i,j=1,2...d_{\beta}$. Thus, since $\sum y_{i}Z_{ij}y_{j}^{\ast}=0$, it
immediately follows that%

\[
\sum y_{i}V_{ij}y_{j}^{\ast}=0,
\]
for every $Y$ as chosen.

In particular, $\sum y_{i}V_{ij}^{pq}y_{j}^{\ast}=0$, for all
$p,q=1,2...d_{\alpha}$, where $V_{ij}^{pq}$ is the entry $pq$ of the
$d_{\alpha}\times d_{\alpha}$ matrix $V_{ij}$. Hence, $\sum y_{i}(\sum
_{p=1}^{d_{\alpha}}V_{ij}^{pp})y_{j}^{\ast}=0.$ Let $d_{ij}=\sum
_{p=1}^{d_{\alpha}}V_{ij}^{pp}$. Therefore, if $Y\in B_{2}^{\delta}(u^{\beta
})$ is as before and $D=\sum_{i,j=1}^{d_{\beta}}d_{ij}\otimes m_{ij}^{\beta}$,
we have $YDY^{\ast}=0$. By Remark \ref{rem2.1} (4) the matrices $Y\in
B_{2}^{\delta}(u^{\beta})$ that have only one non zero row, span
$B_{2}^{\delta}(u^{\beta})$ linearly. Therefore, $YDY^{\ast}=0$ for every
$Y\in B_{2}^{\delta}(u^{\beta})$. Since $Z\geq0$ it follows that $V\geq0$ and
so $D\geq0$. Therefore $YD=0$ for every $Y\in B_{2}^{\delta}(u^{\beta})$.
Notice that $V=\sum_{i,j}V_{ij}\otimes m_{ij}^{\beta}$ satisfies Formula
\ref{formula2} with $u^{\alpha}$ replaced by $u_{1}^{\alpha}$. This fact will
be used in the proof of the next Claim.

Claim: $D\in(B\otimes\mathcal{M}_{d_{\beta}})^{\delta_{u^{\beta}}}$.

The proof of the claim will be achieved in two steps:

Step 1: We prove that%

\begin{equation}
\label{eqn}d_{ij}\otimes1_{A}=\sum_{p=1}^{d_{\alpha}}[(1_{B}\otimes
u_{1}^{\alpha})^{\ast}(V_{ij}\otimes1_{A})(1_{B}\otimes u_{1}^{\alpha})]_{qq},
\end{equation}
where $[(1_{B}\otimes u_{1}^{\alpha})^{\ast}(V_{ij}\otimes1_{A})(1_{B}\otimes
u_{1}^{\alpha})]_{qq}$ denotes the entry $qq$ of the matrix $(1_{B}\otimes
u_{1}^{\alpha})^{\ast}(V_{ij}\otimes1_{A})(1_{B}\otimes u_{1}^{\alpha})$,
$q=1,2...d_{\alpha}$. Tedious but straightforward calculations show that the
right hand side of the above formula is%

\begin{align*}
&  \sum_{q=1}^{d_{\alpha}}[(1_{B}\otimes u_{1}^{\alpha})^{\ast}(V_{ij}%
\otimes1_{A})(1_{B}\otimes u_{1}^{\alpha})]_{qq}\\
&  = \sum_{q}[\sum_{p,n,r,s,l,k}V_{ij}^{pn}\otimes m_{rs}^{\alpha}%
m_{pn}^{\alpha}m_{lk}^{\alpha}\otimes(u_{1}^{\alpha})_{sr}^{\ast}%
(u_{1}^{\alpha})_{lk}]_{qq}\\
&  =\sum_{q}[\sum_{p,n,r,s}V_{ij}^{pn}\otimes\delta_{sp}\delta_{nl}%
m_{rk}^{\alpha}\otimes(u_{1}^{\alpha})_{sr}^{\ast}(u_{1}^{\alpha})_{lk}%
]_{qq}\\
&  =\sum_{q}\sum_{p,n,r,s}\delta_{qr}\delta_{qk}V_{ij}^{pn}\otimes
(u_{1}^{\alpha})_{pr}^{\ast}(u_{1}^{\alpha})_{nk}\\
&  =\sum_{p,n}V_{ij}^{pn}\otimes(\sum_{q}(u_{1}^{\alpha})_{pq}^{\ast}%
(u_{1}^{\alpha})_{nq})=\sum_{p,n}V_{ij}^{pn}\otimes\delta_{pn}1_{A}%
\end{align*}

This last equality holds because we assumed that $\overline{u_{1}^{\alpha}}$
is unitary. Therefore:%

\[
\sum_{q=1}^{d_{\alpha}}[(1_{B}\otimes u_{1}^{\alpha})^{\ast}(V_{ij}%
\otimes1_{A})(1_{B}\otimes u_{1}^{\alpha})_{qq}=\sum_{p,n}V_{ij}^{pn}%
\otimes\delta_{pn}1_{A}=\sum V_{ij}^{pp}\otimes1_{A}=d_{ij}\otimes1_{A},
\]
and the Step 1 is proven.

Step 2: Proof of Claim. We have to prove that:%

\begin{equation}
\label{formula5}\delta_{13}(\sum_{i,j}d_{ij}\otimes m_{ij}^{\beta}%
)=(1_{B}\otimes u^{\beta})^{\ast}(\sum_{i,j}d_{ij}\otimes m_{ij}^{\beta
}\otimes1_{A})(1_{B}\otimes u^{\beta})
\end{equation}

We will evaluate separately the right and left hand sides of Formula
\ref{formula5} and show that they are the same. First, the right hand side:%

\begin{align}
&  (1_{B}\otimes u^{\beta})^{\ast}(\sum_{i,j}d_{ij}\otimes m_{ij}^{\beta
}\otimes1_{A})(1_{B}\otimes u^{\beta})\label{formulan-1}\\
&  =\sum_{q,p,i,j,u,v}(1_{B}\otimes m_{qp}^{\beta}\otimes u_{pq}^{\beta\ast
})(d_{ij}\otimes m_{ij}^{\beta}\otimes1_{A})(1_{B}\otimes m_{uv}^{\beta
}\otimes u_{uv}^{\beta})\nonumber\\
&  =\sum_{q,p,i,j,u,v}d_{ij}\otimes m_{qp}^{\beta}m_{ij}^{\beta}m_{uv}^{\beta
}\otimes u_{pq}^{\beta\ast}u_{uv}^{\beta}\nonumber\\
&  =\sum_{q,p,i,j,u,v}d_{ij}\otimes\delta_{pi}\delta_{ju}m_{qv}^{\beta}\otimes
u_{pq}^{\beta\ast}u_{uv}^{\beta}=\sum_{q,i,j,v}d_{ij}\otimes m_{qv}^{\beta
}\otimes u_{iq}^{\beta\ast}u_{jv}^{\beta}\nonumber
\end{align}

Next we will calculate the left hand side of Formula \ref{formula5}. As
noticed above,by multiplying Formula \ref{formula2} above by $1_{B}\otimes
S^{\ast}\otimes I_{d_{\beta}}\otimes1_{A}$ to the left and by $1_{B}\otimes
S\otimes I_{d_{\beta}}\otimes1_{A}$ to the right, and if we denote $V=\sum
V_{ij}\otimes m_{ij}^{\beta}$, we get%

\[
\delta_{14}(V)=(1_{B}\otimes(u_{1}^{\alpha}\odot u^{\beta})^{\ast}%
)(V\otimes1_{A})(1_{B}\otimes(u_{1}^{\alpha}\odot u^{\beta})).
\]

Therefore:%

\begin{align*}
\delta_{14}(V) &  =\sum_{r,s,k,l,i,j,p,q,t,u,v,w}(1_{B}\otimes m_{rs}^{\alpha
}\otimes m_{kl}^{\beta}\otimes u_{lk}^{\beta\ast}(u_{1}^{\alpha})_{sr}^{\ast
})(V_{ij}^{pq}\otimes m_{pq}^{\alpha}\otimes m_{ij}^{\beta}\otimes1_{A}%
)\times\\
&  (1_{B}\otimes m_{tu}^{\alpha}\otimes m_{vw}^{\beta}\otimes(u_{1}^{\alpha
})_{tu}u_{vw}^{\beta})\\
&  =\sum V_{ij}^{pq}\otimes\delta_{sp}\delta_{qt}m_{ru}^{\alpha}\otimes
\delta_{li}\delta_{jv}m_{kw}^{\beta}\otimes u_{lk}^{\beta\ast}(u_{1}^{\alpha
})_{sr}^{\ast}(u_{1}^{\alpha})_{tu}u_{vw}^{\beta}\\
&  =\sum V_{ij}^{pq}\otimes m_{ru}^{\alpha}\otimes m_{kw}^{\beta}\otimes
u_{ik}^{\beta\ast}(u_{1}^{\alpha})_{pr}^{\ast}(u_{1}^{\alpha})_{qu}%
u_{jw}^{\beta}%
\end{align*}

Hence, if $k=i_{0}$ and $w=j_{0}$ we get:%

\[
\delta_{14}(V_{i_{0}j_{0}}\otimes m_{i_{0}j_{0}}^{\beta})=\sum_{p,q,r,u,i,j}%
V_{ij}^{pq}\otimes m_{ru}^{\alpha}\otimes m_{i_{0}j_{0}}^{\beta}\otimes
u_{ii_{0}}^{\beta\ast}(u_{1}^{\alpha})_{pr}^{\ast}(u_{1}^{\alpha}%
)_{qu}u_{jj_{0}}^{\beta}%
\]

and if $r=u=l$,%

\[
\delta_{14}(V_{i_{0}j_{0}}^{ll}\otimes m_{ll}^{\alpha}\otimes m_{i_{0}j_{0}%
}^{\beta})=\sum_{p,q,i,j}V_{ij}^{pq}\otimes m_{ll}^{\alpha}\otimes
m_{i_{0}j_{0}}^{\beta}\otimes u_{ii_{0}}^{\beta\ast}(u_{1}^{\alpha}%
)_{pl}^{\ast}(u_{1}^{\alpha})_{ql}u_{jj_{0}}^{\beta}.
\]

Therefore:%

\[
\delta_{13}(d_{i_{0}j_{0}}\otimes m_{i_{0}j_{0}}^{\beta})=\sum_{p,q,i,j}%
V_{ij}^{pq}\otimes m_{i_{0}j_{0}}^{\beta}\otimes u_{ii_{0}}^{\beta\ast}%
(\sum_{l=1}^{d_{\alpha}}(u_{1}^{\alpha})_{pl}^{\ast}(u_{1}^{\alpha}%
)_{ql})u_{jj_{0}}^{\beta}%
\]

Since $\overline{u_{1}^{\alpha}}$ is a unitary representation, we have
$\sum_{l=1}^{d_{\alpha}}(u_{1}^{\alpha})_{pl}^{\ast}(u_{1}^{\alpha}%
)_{ql}=\delta_{pq}$ where, as usual, $\delta_{pq}$ is the Kronecker symbol. Hence:%

\[
\delta_{13}(d_{i_{0}j_{0}}\otimes m_{i_{0}j_{0}}^{\beta})=\sum_{i,j}%
d_{ij}\otimes m_{i_{0}j_{0}}^{\beta}\otimes u_{ii_{0}}^{\beta\ast}u_{jj_{0}%
}^{\beta}%
\]

Thus:%

\begin{equation}
\label{formulan}\delta_{13}(D)=\sum_{i,j,i_{0}j_{0}}d_{ij}\otimes
m_{i_{0}j_{0}}^{\beta}\otimes u_{ii_{0}}^{\beta\ast}u_{jj_{0}}^{\beta}%
\end{equation}

Formulas \ref{formulan} and \ref{formulan-1} show that the Claim is true.

Since $\beta\in Sp(\delta),D\in(B\otimes\mathcal{M}_{d_{\beta}})^{\delta
_{u^{\beta}}}$ and $YD=0$ for every $Y\in B_{2}^{\delta}(u^{\beta})$, it
follows that $D=0$. This means in particular that all the diagonal entries of
the matrix $V$ are equal to $0$. Since $V\geq0$, it follows that $V=0$ and
thus $Z=0$. Therefore $linspan\{(X\odot Y)^{\ast}(X\odot Y)|X\in B_{2}%
^{\delta}(u^{\alpha}),Y\in B_{2}^{\delta}(u^{\beta})\}$ is an essential ideal
of $(B\otimes\mathcal{M}_{d_{\alpha}d_{\beta}})^{\delta_{u^{\alpha}\odot
u^{\beta}}}$ as claimed.

Let now $u^{\alpha}\odot u^{\beta}\widetilde{=}\sum_{i}^{\oplus}u^{\rho_{i}}$
where $\rho_{i}$ are irreducible. Then, by Remark \ref{odot}(1) above, it
follows that $(B\otimes\mathcal{M}_{d_{\alpha}d_{\beta}})^{\delta_{u^{\alpha
}\odot u^{\beta}}}$ is spatially $\ast$-isomorphic to $\sum^{\oplus}%
(B\otimes\mathcal{M}_{d_{\rho_{i}}})^{\delta_{\rho_{i}}}$. Thus, since
$B_{2}^{\delta}(u^{\alpha}\odot u^{\beta})$ is spatially isomorphic to
$\sum^{\oplus}B_{2}^{\delta}(\rho_{i})$ (Remark \ref{odot}(2) above), it
follows that $B_{2}^{\delta}(\rho_{i})^{\ast}B_{2}^{\delta}(\rho_{i})$ is an
essential ideal of $(B\otimes\mathcal{M}_{d_{\rho_{i}}})^{\delta_{\rho_{i}}}$,
for all $i$. Therefore $\rho_{i}\in Sp(\delta)$, for every $i$. Thus
$Sp(\delta|_{C})$ is closed under tensor products for every $C\in
\mathcal{H}^{\delta}(B)$ the lemma is proven.
\end{proof}

We can now state:

\begin{proposition}\label{connesclosed}
The Connes spectrum, $\Gamma(\delta)$ is closed under tensor products.
\end{proposition}

\end{document}